\documentclass[12pt]{amsart}
\usepackage{amsmath,amsfonts,euscript,amscd,amsthm,amssymb,upref,graphicx,color}
\usepackage{comment}

\input xy
\xyoption{all}

\theoremstyle{plain}
\newtheorem{theorem}{Theorem}
\swapnumbers

\newtheorem{proposition}[subsection]{Proposition}
\newtheorem{lemma}[subsection]{Lemma}

\theoremstyle{definition}
\newtheorem{definition}[subsection]{Definition}

\newtheorem{nothing*}[subsection]{}

\newcommand{\rien}[1]{}

\renewcommand{\epsilon}{\varepsilon}
\renewcommand{\phi}{\varphi}

\begin{document}
\renewcommand{\baselinestretch}{1.07}

\title[Symplectic Completion and involutive lines]
{Polynomial Completion of Symplectic Jets and surfaces containing involutive lines}

\author[E. L\o w]{Erik L\o w}
\address{Matematisk Institutt \\ Universitetet i Oslo
\\Postboks 1053 Blindern
\\ NO-0316 Oslo, Norway}
\email{elow@math.uio.no}
\email{erlendfw@math.uio.no}

\author[J.V. Pereira]
{Jorge V. Pereira}
\address{IMPA, Estrada Dona Castorina, 110, Horto, Rio de Janeiro,
Brasil}\email{jvp@impa.br}

\author[H. Peters]{Han Peters}
\address{Korteweg-de Vries Institute for Mathematics\\
University of Amsterdam\\
Science Park 904\\
1098 XH Amsterdam\\
The Netherlands}
\email{h.peters@uva.nl}

\author[E.F. Wold]{Erlend F. Wold}

\date{January 14, 2007}

\begin{abstract}
Motivated by work of Dragt and Abell on accelerator physics, we study the completion of symplectic jets by polynomial maps of low degrees.  We use Anders\'en-Lempert Theory to prove that symplectic completions always exist, and we prove the degree bound conjectured by Dragt and Abell in
the physically relevant cases.  However, we disprove the degree bound for $3$-jets in dimension 4.  This follows from the fact that if $\Sigma$ is the
disjoint union of $r=7$ involutive lines in $\mathbb P^3$, then $\Sigma$ is contained in a degree $d=4$  hypersurface, \emph{i.e.},
the restriction morphism $\iota:H^0(\mathbb P^3,\mathcal O(4))\rightarrow H^0(\Sigma,\mathcal O(4))$
has a nontrivial kernel (Todd).  We give two new proofs of this fact, and finally we show that
if $(r,d)\neq (7,4)$ then the map $\iota$ has maximal rank.

\end{abstract}

\maketitle \vfuzz=2pt

\vfuzz=2pt

\vskip 1cm

\section{Introduction}

Throughout this article we let $\omega$ denote the standard symplectic form on
$\mathbb C^{2n}$, \emph{i.e.}, $\omega(z)=\sum_{j=1}^ndz_j\wedge dz_{n+j}$.
If $\Omega\subset\mathbb C^{2n}$ is a domain, and $F:\Omega\rightarrow\mathbb C^{2n}$
is an analytic map, we say that $F$ is a symplectomorphism if
$F^*\omega=\omega$.  If $a\in\Omega$ and $F$ merely satisfies $(F^*\omega-\omega)(z)=O(\|z-a\|^d)$, then
we say that $F$ is symplectic to order $d$ at $a$. If $F$ is also a degree $d$ polynomial map,
then we call $F$ a symplectic $d$-jet at $a$.  In this article we consider the completion of symplectic jets

A linear subspace $\Lambda\subset\mathbb C^{2n}$ is
said to be Lagrangian if it has dimension $n$ and  $\omega|_{\Lambda}=0$. We will say that  a linear subspace $H\subset\mathbb P^{2n-1}$
is  involutive if $\pi^*(H)\cup\{0\}$ is Lagrangian, where $\pi$ is the projection
$\pi:\mathbb C^{2n}\setminus\{0\}\rightarrow\mathbb P^{2n-1}$.
For integers $(n,d)$ we set $N(n,d):=$ ${n+d-1}\choose{d}$, the dimension of the space of homogeneous polynomials of degree $d$ in $\mathbb C^{n}$.
We also write $M(2n,d):=\lceil N(2n,d)/N(n,d)\rceil$.
We denote by $Aut_{Sp}\mathbb C^{2n}$
the group of symplectic automorphisms of $\mathbb C^{2n}$.

\medskip

In this article we will prove the following results, motivated by accelerator physics.

\begin{theorem}\label{thm1}
Let $P$ be a symplectic $d$-jet at the origin in $\mathbb C^{2n}$.  Then
there exists a polynomial map $F\in Aut_{Sp}\mathbb C^{2n}$ such that
$(F-P)(z)=O(\|z\|^d)$.
\end{theorem}

\begin{theorem}\label{thm2}
Let $P$ be a symplectic $d$-jet at the origin in $\mathbb C^{6}$ with
$2\leq d\leq 11$.  Then
there exists a polynomial map $F\in Aut_{Sp}\mathbb C^{6}$ with
$deg(F)\leq d^{M(6,d+1)}$  such that
$(F-P)(x)=O(\|z\|^d)$.
\end{theorem}
The proof of Theorem \ref{thm1} uses Anders\'{e}n-Lempert theory, and
follows closely the work of
Forstneri\v{c} \cite{For1}.  Theorem \ref{thm2} follows from the same proof
and the following:

\begin{theorem}\label{thm3}
Let $3\leq d\leq 12$ and let $\Sigma$ be a disjoint union of
$r=M(6,d)$ involutive
planes in general position in
$\mathbb P^5$. Then the restriction
map
$$
H^0(\mathbb P^5,\mathcal O(d))\rightarrow H^0(\Sigma,\mathcal O(d))
$$
has maximal rank.
\end{theorem}
Recall that the global sections of $\mathcal O(d)$ correspond to homogeneous polynomials of degree $d$. Theorem \ref{thm3} proves a conjecture by Dragt and Abell \cite{DragtAbell} on bounding the degree of
symplectic completions in the cases relevant for accelerator physics. The proof of Theorem \ref{thm3} uses computer algebra and is described in Section \ref{section:three}. We note that the proofs of Theorems \ref{thm1}, \ref{thm2} and \ref{thm3} are constructive, and give algorithms for finding the symplectic completions.

However, the conjecture does not hold in general:
\begin{theorem}(Todd)\label{thm4}
Let $\Sigma$ be a collection of $7$ disjoint involutive lines in $\mathbb P^3$.
Then $\Sigma$ is contained in a quartic.
\end{theorem}
We give two new proofs of this result. In Section \ref{section:three} we prove Theorem \ref{thm4} by using computer algebra, and in Section \ref{section:five} by  exhibiting  the sought  quartic as the tangency locus of two foliations tangent to the contact structure and the 7 involutive lines.

Finally we give a complete version of Theorem \ref{thm3} in dimension $3$:
\begin{theorem}\label{thm5}
Let $\Sigma$ be a disjoint union of $r$ involutive lines in general
position in $\mathbb P^3$ and $d\geq 0$ be an integer.  If $(r,d)\neq (7,4)$
then the restriction map
$$
H^0(\mathbb P^3,\mathcal O(d))\rightarrow H^0(\Sigma,\mathcal O(d))
$$
has maximal rank.
\end{theorem}

Our proof in Section \ref{section:six} follows very closely the work of Hartshorne and Hirschowitz \cite{HH1} who proved the result for non-involutive lines.

\medskip

\noindent {\bf Motivation from Accelerator Physics.} Let us briefly outline the motivation from accelerator physics for the problems considered in this article, for more details see \cite{DragtAbell}. Consider a circular particle accelerator where two groups of particles circulate a large number of times in opposite directions, before hitting each other at a given location. In order to control this collision, a precise understanding of the orbits of the particles is needed. Instead of following the continuous flows of these orbits, we can take a cross-section of the accelerator and compute the map $G$ that, given a current intersection of the orbits with this cross-section, computes the next intersection. It turns out that this is a symplectic map in $6$ real variables.

Let us assume that the map $G$ has an ideal orbit. After rescaling we may assume that $G(0) = 0$. We also assume that $G$ is given by a convergent power series expansion. We can then try to measure the first so many coefficients of this expansion. In practice the power series of $G$ is measured up to and including degree $11$. Denote the degree $11$-jet of $G$ by $P$. Then in general $P$ is only symplectic to order $11$, and the iterative behavior of $P$ can be distinctly different from the behavior of $G$. The idea is now to find a new symplectomorphism $F$ that has exactly the same $11$-jet as $P$. It turns out that even though the higher order terms of $F$ are unrelated to the higher order terms of $G$, the iterative behavior of $F$ may approximate the iterative behavior of $G$ better than $P$ does.

The goal is therefore to not only prove that $P$ has a symplectic completion, but also to find a method for actually finding a symplectic completion of $P$, and preferably a completion with a relatively small complexity so that a large number of iterations of $F$ can be computed. In this article we find symplectic completions given by polynomials with relatively small degrees.

\medskip

{\bf Acknowledgement.}
We thank Meike Wortel for help with the Mathematica code.
The third author was supported by a SP3-People Marie Curie Actionsgrant in the project Complex Dynamics (FP7-PEOPLE-2009-RG, 248443).
The fourth author was supported by the NFR grant 209751/F20.

\section{Background: A question of Dragt and Abell}

Throughout this article we denote vectors in $\mathbb C^{2n}$ by boldface letters ${\bf a}$ when we want to emphasize that we do not regard them as variables, and we let ${\bf a}\cdot z$ denote the product ${\bf a}\cdot z=\sum_{j=1}^{2n}{\bf a_j} \cdot z_j$.
We let $J$ denote the symplectic involution
$$
J(z)=(-z_{n+1},...,-z_{2n},z_1,...,z_n).
$$
We denote by $V_{n,d}$ the vector space of $d$-homogenous polynomials in $n$ variables.

Recall the definition of a \emph{kick-map}; a holomorphic automorphism of the simple form
\begin{equation}\label{kick}
G(z)=(z_1+g_1(z_{n+1},...,z_{2n}),\cdot\cdot\cdot, z_n+g_n(z_{n+1},...,z_{2n}),z_{n+1},\cdot\cdot\cdot ,z_{2n}).
\end{equation}
If all the $g_j$'s are $d$-homogenous polynomials we call $G$ a $d$-homogenous kick-map.
A $d$-homogenous kick-map is symplectic if and only if the matrix
\begin{equation}
(\partial g_j/\partial z_{n+i})_{1\leq i,j\leq n}
\end{equation}
is symmetric, \emph{i.e.}, if the differential form $\sum_{j=1}^ng_j(z)dz_{n+j}$ is closed, hence also exact.  In particular we may identify the space of symplectic kick-maps with the
space of $(d+1)$-homogenous polynomials in $n$ variables.  We can now
try to find symplectic completions by considering conjugations $L\circ G\circ L^{-1}$ of
symplectic kick-maps by symplectic linear maps, and compositions of these.  This reduces
to finding a suitable basis for the space $V_{n,d+1}$ (see Section 4).
Since we will mostly be working with potentials of the $Q_d$'s we will from now on write $d$ instead of $d+1$. \

Related to this, Dragt and Abell pose the following problem \cite{DragtAbell}:
given $d\in\mathbb N$, find the least number
of linear maps $L_j\in Aut_{Sp}\mathbb C^{2n}, j=1,...,M$, such
that any $P\in V_{2n,d}$ can be written as
\begin{equation}\label{sum}
P=\sum_{j=1}^M Q_j\circ L_j,
\end{equation}
where the $Q_j$'s are $d$-homogenous polynomials in the
variables $(z_{n+1},...,z_{2n})$.  They conjecture
that one always can achieve (\ref{sum}) with $M=M(2n,m)$,
and they note that in the physically relevant cases the dimension $2n$ equals $6$, and the degree $d$ is at most $12$. Recall that $d=12$ corresponds to symplectic jets of degree $11$. Dragt and Abell claim to have found linear maps $(L_j)$ that confirm the conjecture for degrees up to $6$. \

To see how this conjecture is related to Theorem \ref{thm3} we use the following fundamental fact.

\begin{lemma}\label{fundamental}
Let $\Sigma\subset\mathbb C^n$ be any subset, and let $d\in\mathbb N$.   Then the elements in
$$
E=\{({\bf a}\cdot z)^d:\bf a\rm\in\Sigma\}
$$
form a basis for the vector space $V_{n,d}$,
if and only if there does not exist a $d$-homogenous polynomial $P$ which vanishes on
$\Sigma$.  In particular, if $\Sigma=\mathbb C^n$ then $E$ spans $V_{n,d}$.
\end{lemma}

\begin{proof}
The elements $z^{\alpha}$ with $|\alpha|=d$, is a basis for the vector space $V_{n,d}$.
Expressed in this basis we have that $({\bf a}\cdot z)^d=\sum_{\alpha}b(\alpha){\bf a}^\alpha z^\alpha$
where
the coefficients $b(\alpha)$ are the multi-binomial coefficients
corresponding to $\alpha$.  Assume that $\Lambda$ is a non-zero linear form that anilates
all elements of the form $({\bf a}\cdot z)^m$, and write
\begin{equation*}
\Lambda ({\bf a}\cdot z)^m = \sum_\alpha \lambda_\alpha \cdot b(\alpha) {\bf a}^\alpha \equiv 0.
\end{equation*}
Considering ${\bf a}^\alpha$ as a basis for $V_{n,d}$ we get a $d$-homogenous polynomial that
vanishes on $\Sigma$.  Conversely, any such polynomial gives rise to a linear form on $V_{n,d}$
which vanishes on $E$.
\end{proof}

Now assume that $\Sigma$ is a disjoint union of $M(2n,d)$ involutive subspaces
of $\mathbb P^{2n-1}$ and assume that the restriction $H^0(\mathbb P^{2n-1},\mathcal O(d))\rightarrow H^0(\Sigma,\mathcal O(d))$ is injective.  Then we think of $\Sigma$ as a collection $\Sigma_1,...,\Sigma_{M(2n,d)}$
of Lagrangian subspaces of $\mathbb C^{2n}$, and according the previous lemma the
set $E$ forms a basis for $V_{n,d}$.  For each $\Sigma_j$ pick $L_j\in Aut_{Sp}\mathbb C^{2n}$
such that $(L_j^{-1})^T$ maps  $\Sigma_j$ to $\{z\in\mathbb C^{2n}:z_1= \cdots = z_n=0\}$.
Now the elements $({\bf a}\cdot z)^d=((L_j^{-1})^T{\bf a}\cdot L_jz)^m$, with ${\bf a}\in\Sigma_j$
and $j=1,...,M(2n,d)$, form a basis for $V_{2n,d}$.

\section{Proof of Theorem \ref{thm1}} \label{section:two}

We start by making the following observation.

\begin{lemma}\label{condition}
Consider a polynomial map $P:\mathbb C^{2n}\rightarrow\mathbb C^{2n}$,
$$
P(z)=z + P_d(z) + O(\|z\|^{d+1}),
$$
Then $(P^*\omega-\omega)(z)=O(\|z\|^{d})$
if and only if $d(P_m\lrcorner\omega)=0$, \emph{i.e.}, if and only if $P_d$ is Hamiltonian regarded
as a vector field on $\mathbb C^{2n}$.
\end{lemma}

\begin{proof}
We have
\begin{eqnarray*}
     P^* \omega (u,v) &=& \omega (P_{*}u , P_{*}v)\\
               &=& \omega (u+(P_m)_*u + O(\|z\|^d), v+(P_d)_*v + O(\|z\|^{d}))\\
               &=& \omega (u,v) + \omega (u,(P_d)_*v) + \omega ((P_d)_*u,v)+ O(\|z\|^d)\\
               &=& \omega (u,v) - v[(P_d \lrcorner \omega)(u)] +
                 u[(P_d \lrcorner \omega)(v)] + O(\|z\|^d)\\
               &=& \omega (u,v) + d(P_d \lrcorner \omega)(u,v) + O(\|z\|^d)\\
\end{eqnarray*}
The last equality is the
well-known formula for the differential of a 1-form, and the next to last equality is
an easy computation. The second term is of degree $d-1$ in $z$ and must therefore be
zero, \emph{i.e.}, $d(P_d \lrcorner \omega) = 0$
\end{proof}

We will now prove the key lemma for proving Theorem \ref{thm1}.

\begin{lemma}\label{basis}
The vector space of $d$-homogenous Hamiltonian vector
fields on $\mathbb C^{2n}$ is spanned by vector fields of the form $(J{\bf a}\cdot z)^d\cdot{\bf a}$.
\end{lemma}

\begin{proof}
Let $P_d(z)$ be a hamiltonian vector field.
Then there exists a $(d+1)$-homogenous polynomial $Q(z)$
such that $P_d(z)=J\nabla Q(z)$.   By Lemma \ref{fundamental}
we may write
$$
Q(z)=\sum_{j=1}^N c_j\cdot ({\bf a}^j\cdot z)^{m+1}
$$
from which the result follows.
\end{proof}

\medskip

\emph{Proof of Theorem \ref{thm1}:}
We are given a symplectic $d$-jet $P(z)$.   We will prove the result by
induction on $k$ for $2\leq k\leq d$. Since $DP(0)\in Aut_{Sp}\mathbb C^{2n}$
we may match $P$ to order one.

Assume next that we have found $F_k\in Aut_{Sp}\mathbb C^{2n}$
with $(F_k-P)(z)=O(\|z\|^{k})$ for $2\leq k<d$.  Write
$$
G_k(z):=P\circ F_k^{-1}(z)=z + P_{k+1}(z) + O(\|z\|^{k+2}).
$$
Since $G_k$ is a symplectic $d$-jet and $d \ge k+1$, it follows by Lemma \ref{condition}
that $P_{k+1}$ is Hamiltonian, and by Lemma \ref{basis} we may write
$$
P_{k+1}(z)=\sum_{j=1}^N c_j\cdot(J{\bf a}^j\cdot z)^{m}\cdot{\bf a}^j.
$$
Define $S_j(z):=z + c_j\cdot (J{\bf a}^j\cdot z)^m\cdot{\bf a}^j$.
Then $S_j\in Aut_{Sp}\mathbb C^n$ for all $j$, and the composition $H_{k+1}:=S_N\circ\cdot\cdot\cdot\circ S_1$
matches $G_k$ to order $k+1$.   Then $F_{k+1}:=H_{k+1}\circ F_k$
matches $P$ to order $k+1$.
$\hfill\square$

\section{Proofs of Theorems 2, 3 and 4 using computer algebra} \label{section:three}

Let $P(z)=z+P_d(z)$ be a homogenous symplectic $d$-jet at the origin.
The algorithm in the proof of Theorem \ref{thm1} is likely to
produce a symplectic completion $F$ whose degree is too large to be useful in practice. More specifically, the degree of
$F$ will depend on the choice of basis $\{(J{\bf a}^j \cdot z)^{d+1}\}_{j=1}^{N(2n,d+1)}$.
To illustrate this, consider two vectors ${\bf a}$ and ${\bf b}$, and the composition
of two shears:
\begin{equation}\label{composition}
S_b\circ S_a(z)= z + (J{\bf a} \cdot z)^{d}\cdot{\bf a} + (J{\bf b} \cdot (z + (J{\bf a} \cdot z)^{d}\cdot{\bf a}))^{d}\cdot{\bf b}
\end{equation}
Composing more maps, we see that we run the risk of producing a degree $d^{N(2n,d+1)}$ polynomial map while
trying to match a $d$-jet. However, inspecting (\ref{composition}) we note that if
$J{\bf b} \cdot {\bf a}=0$, \emph{i.e.}, if ${\bf a}$ and ${\bf b}$ lie in
a \emph{Lagrangian} subspace of $\mathbb C^{2n}$, then the degree of
$S_2\circ S_1$ is $d$ (recall that an $n$-dimensional subspace $\Lambda\subset\mathbb C^{2n}$
is Lagrangian if $\omega|_{\Lambda}=0$, or, equivalently, if $J{\bf a}\cdot {\bf b}=0$ for all ${\bf a},{\bf b}\in\Lambda$).  Hence, to keep the growth of degree down when composing
our shear maps, we should choose the maximal number of vectors ${\bf a}^j$ from the same Lagrangian subspaces.   For any given Lagrangian $\Lambda\subset\mathbb C^{2n}$, it follows from Lemma \ref{fundamental} that we may find vectors ${\bf a}^j \in\Lambda$, $j=1,...,N(n,d+1)$, such that the vectors $({\bf a}^j \cdot z)^{d+1}$ are linearly independent.
It follows that we need at least $M(2n,d+1)=\lceil N(2n,d+1)/N(n,d+1)\rceil$ Lagrangian
subspaces. Note that $M(2n, d+1)$ is exactly the number appearing in the Conjecture of Dragt and Abell.

For simplicity of notation we will from now on write $d$ instead of $d+1$. This will not create confusion as we will no longer need to consider the jet $P$.

\begin{definition}
If the $d$-homogenous polynomials $\{({\bf a}\cdot z)^{d}\}_{{\bf a}\in\Lambda_j,1\leq j\leq M(2n,d)}$
span the vector space of all $d$-homogenous polynomials then we say that the Lagrangian subspaces $(\Lambda_j)$ \emph{span degree $d$}.
\end{definition}

We first discuss our approach to proving  Theorem \ref{thm3} for general dimension $n$, and later restrict to $2n=6$.
Choose $M(2n,d)$ Lagrangian subspaces $\Lambda_j\subset\mathbb C^{2n}$
at random, and for each $j$ choose vectors ${\bf a}^{k,j} \in \Lambda_j$, for $1\leq k\leq N(n,d)$.
Expand the $d$-homogenous polynomials $({\bf a}^{k,j}\cdot z)^d$ in
a convenient basis, and call the corresponding vectors $v_{k,j}$.   We form
the matrix whose rows are the vectors $v_{k,j}$ and then compute the rank.
If the rank turns out to be $N(2n,d)$ we have proved Theorem \ref{thm3} for degree $d$.
If however the rank turns out to be less than $N(2n,d)$ we cannot draw conclusions;
it might merely be caused by a bad choice
of Lagrangian spaces and vectors.

If two Lagrangian subspaces $\Lambda_1$ and $\Lambda_2$ intersect non-trivially then it is clear that the matrix formed by the vectors $({\bf a}^{k,j} \cdot z)^d$ cannot have maximal rank. Therefore we should consider collections of Lagrangian subspaces $\Lambda_1, \ldots, \Lambda_k$ that satisfy
\begin{equation*}
\Lambda_i \cap \Lambda_j = \{0\},
\end{equation*}
for all $i \neq j$. It turns out that there exists a convenient form for such collections.

\begin{lemma}\label{normalform}
Let $\Lambda_1, \ldots, \Lambda_k$ be Lagrangian subspaces of $\mathbf{C}^{2n}$ with
the property that
\begin{equation*}
\Lambda_i \cap \Lambda_j = \{0\},
\end{equation*}
for any $ i \neq j$. Then, after a suitable linear symplectic change of coordinates
the Lagrangian subspaces $(\Lambda_j)$ are spanned by the matrices
\begin{equation*}
L_1 = [I, 0], L_2 = [0, I],
\end{equation*}
and
\begin{equation*}
L_j = [I, A_j]
\end{equation*}
for $j = 3, \ldots k$. Moreover, the $n\times n$ matrices $A_j$ are symmetric, have
non-zero determinant, and the same holds true for the matrices $(A_i - A_j)$ for distinct $i, j \ge 3$.
\end{lemma}
\begin{proof}
First note that we can change coordinates so that $\Lambda_1$ is induced by $[I, 0]$.
Now suppose that $\Lambda_2$ is induced by the matrix $[A, B]$. Using the fact that
the intersection of $\Lambda_1$ and $\Lambda_2$ is $\{0\}$ we get that $\det B \neq 0$.
Hence after changing coordinates with the linear symplectomorphism given by
\begin{equation*}
(z, w) \rightarrow (z - AB^{-1}w,w)
\end{equation*}
the subspace $\Lambda_2$ will be given by the matrix $[0, B]$, or equivalently
$[0, I]$, while $\Lambda_1$ is still given by $[I,0]$.

Using that $\Lambda_j \cap \Lambda_2 = \{0\}$ for $ j \ge 3$ immediately gives that
$\Lambda_j$ is spanned by a matrix of the form $[I, A_j]$. The fact that the subspace
is Lagrangian forces the matrix $A_j$ to be symmetric. The determinant of $A_j$ must be
non-zero or else $\Lambda_j$ intersects $\Lambda_1$ in at least a line. Similarly
$\det (A_j - A_i)  \neq 0$ or else $\Lambda_j$ and $\Lambda_i$ intersect in at least a line.
\end{proof}

Working with Lagrangian subspaces that are generated by matrices of this form has several advantages. One is that, at least for small degree $d$ and dimension $2n$, we can explicitly describe by computer for which matrices $(A_j)$ the Lagrangian subspaces can generate a matrix $T$ of maximal rank. When $2n = 4$ and $d = 3$ the matrix $T$ is a $20 \times 20$ square matrix. A symbolic computation of its determinant in Mathematica gives the following.

\begin{lemma}
Consider five Lagrangian subspaces $(\Lambda_j)$ in $\mathbb{C}^4$ generated by $[I, 0]$, $[0, I]$, $[I, A]$, $[I, B]$ and $[I, C]$, where $A, B, C$ are symmetric $2 \times 2$ matrices that satisfy the conditions in Lemma \ref{normalform}. Let us write $A = (a_{kl})$, $B = (b_{kl})$, and $C = (c_{kl})$. Then $\Lambda_1 , \ldots , \Lambda_5$ span degree $3$ if and only if
\begin{equation*}
\det \left[ \begin{matrix}
a_{11} & b_{11} & c_{11}\\
a_{12} & b_{12} & c_{12}\\
a_{22} & b_{22} & c_{22}
\end{matrix} \right] \neq 0.
\end{equation*}
\end{lemma}

We were unable to find similar formulas for larger degrees. We shall see below that there is no collection of $M(4,4)$ Lagrangian subspaces that spans degree $4$. For degrees $5$ and larger our computer was unable to run the symbolic computation.

\subsection{Proof of Theorem \ref{thm2}}

An advantage of Lagrangian subspaces of the form described in Lemma \ref{normalform} is that it allows us to easily work with integer coefficients. If there exists a collection $(\Lambda_j)$ that spans degree $d$, then degree $d$ is spanned for collections $(\Lambda_j)$ in a Zariski open subset. Hence maximal rank will also be attained for matrices $(A_j)$ with integer coefficients. Working with integer coefficients will turn out to be very helpful.

Let $max \in \mathbb{N}$. Using a computer we randomly generate symmetric matrices $(A_j)$ for $j = 1, \ldots , M(2n, d)$, where each upper diagonal entry is chosen independently from the interval $[1, max] \subset \mathbb{N}$. In order to choose the vectors ${\bf a}^{j, k}$ from the Lagrangian subspace $\Lambda_j$, we randomly choose an $n \times N(n, d)$ matrix $X$, again with each entry chosen independently from $[1,max]$, and define
\begin{equation*}
\left[\begin{matrix}
{\bf a}^{j,1}\\
\vdots\\
{\bf a}^{j,M}
\end{matrix}
\right] = X \cdot \left[I,  A_j\right].
\end{equation*}
We write each homogeneous polynomial $({\bf a}^{j,k} \cdot z)^d$ as a vector with respect to the same monomial basis to obtain the matrix
\begin{equation}\label{matrixT}
T = T(\{{\bf a}^{j,k}\} = \left[
\begin{matrix}
\vdots\\
({\bf a}^{j,k} \cdot z)^d\\
\vdots
\end{matrix}
\right],
\end{equation}
where $T$ again has integer coefficients. Hence, at least in theory, a computer is able to compute the rank of $T$. If this rank is maximal then we have found a collection $(\Lambda_j)$ that spans degree $d$.

We were able to run this program in Mathematica for dimension $2n = 6$ and $d \le 12$. It turns out that as the degree $d$ grows, not only does the size of $T$ grow, but also the necessary interval $[1, max]$. Hence the coefficients occurring in $T$ grow rapidly, and the program quickly becomes too large to run even on the strongest computer.

We were able to deal with larger degrees by computing the rank of $T$ modulo suitably chosen primes. Again if this rank is maximal then we know for sure that the collection $(\Lambda_j)$ spans degree $d$. This allowed us to find a collection $(\Lambda_j)$ that spans degrees up to $12$, without having to run our program on a special computer. It is likely that by running the program on a supercomputer, spanning collections can be found for slightly larger degrees.

For completeness, we print the Mathematica code used to prove Theorem $3$ for degree $d=12$. The code is similar for lower degrees.

\vspace{.1in}

\noindent\(\pmb{\text{makeL}[\text{max$\_$}] \text{:=} \text{Module}[\{\text{matrix}\},}\\
\pmb{\text{matrix} = \text{Table}[\text{RandomInteger}[\{1,\max \}], \{j, 1, 3\}, \{i, 1, 3\}];}\\
\pmb{\text{matrix}[[1, 2]] = \text{matrix}[[2,1]];}\\
\pmb{\text{matrix}[[1,3]] = \text{matrix}[[3,1]];}\\
\pmb{\text{matrix}[[2,3]] = \text{matrix}[[3,2]];}\\
\pmb{\text{Transpose}[\text{Flatten}[\{\text{IdentityMatrix}[3], \text{matrix}\}, 1]]}\\
\pmb{]}\)

\noindent\(\pmb{\text{tensor}[\{\text{a$\_$}, \text{b$\_$}, \text{c$\_$}, \text{d$\_$}, \text{e$\_$}, \text{f$\_$}\}]\text{:=} }\\
\pmb{\text{Evaluate}[\text{List}\text{@@}\text{Expand}[(a+b+c+d+e+f){}^{\wedge}12]]}\)

\noindent\(\pmb{\text{makematrix}[\text{max$\_$}] \text{:=} \text{Module}[\{\text{matrix2}, \text{matrix3}\},}\\
\pmb{\text{matrix2} =\text{  }\text{Table}[\text{RandomInteger}[\{1,\max \}], \{j, 1,91\}, \{i, 1, 3\}];}\\
\pmb{\text{matrix3} = \text{Flatten}[\text{Table}[\text{matrix2} . \text{makeL}[\max ], \{i, 1,68\}], 1];}\\
\pmb{\text{Thread}[\text{tensor}[\text{matrix3}]]}\\
\pmb{]}\)

\noindent\(\pmb{T = \text{makematrix}[20000];}\)

\noindent\(\pmb{\text{Do}[\text{Print}[\text{MatrixRank}[T,\text{Modulus} \to  \text{Prime}[i]]], \{i, 100, 150\}]}\)

\vspace{.1in}

After running the program for one night the ranks of the first $26$ were computed:
$$
\begin{aligned}
& 6188, 6181, 6180, 6187, 6186, 6183, 6188, 6187, 6186, 6184, 6187, 6185, 6184, 6185, 6185,\\
& 6184, 6186, 6186, 6187, 6187, 6184, 6185, 6184, 6183, 6186, 6185
\end{aligned}
$$
Notice that the required rank 6188 was obtained twice.

\subsection{Proof of Theorem 4}\label{section:four}

While we will give an explicit proof of Theorem \ref{thm4} in the next section, we will now outline how the same result can be obtained using computer algebra. As explained above, Theorem \ref{thm4} follows if it can be shown that the determinant of the matrix $T$ is zero for all choices of the matrices $A_3$ through $A_7$. One easily sees that the determinant of $T$ is a polynomial of degree $50$ in the $15$ independent entries of the symmetric $2\times 2$ matrices $A_3, \ldots , A_7$. Hence to conclude that $P = 0$ it is sufficient to evaluate $P$ on a grid with $51^{15}$ entries. Unfortunately this is far beyond the scope of current-day computers.

The size of the grid can be significantly reduced using the following observations. Most importantly, after dividing by
$$
Q = \prod_{j = 3, \ldots 7} \mathrm{det}(A_j) \prod_{i < j} \mathrm{det}(A_j - A_i),
$$
which is a product of $15$ distinct prime factors, the remaining polynomial $P/Q$ has only degree either $4$ or $5$ in each of the variables separately. Secondly $P$ and $Q$ are invariant under an action of the group $S_2 \times S_5$. Using these observations we constructed an easily programmable grid (on which $Q \neq 0$) with only $489,742,800$ entries, and evaluated $P$ on this grid using a computer. As expected, $P$ turned out to vanish identically, proving Theorem \ref{thm4}.

\section{Seven Involutive Lines}\label{section:five}

It is well-known that every line in $\mathbb P^3$ is contained in a one-parameter family of hyperplanes; that $3$ pairwise disjoint lines in $\mathbb P^3$  determine a unique quadric containing them;
and that $5$ pairwise disjoint lines in $\mathbb P^3$  are contained in a cubic surface if and only if there exists another line intersecting all of them. Less well-known is the following
result of Cayley: given six lines $\ell_1, \ldots, \ell_6$
in general position in $\mathbb P^3$  then there exists a quartic line complex $\mathfrak Q$ (depending on the six lines)  such that the union of $\ell_1, \ldots, \ell_6$ and a seventh line $\ell_7$ is contained in a quartic surface if and only $\ell_7$ belongs to $\mathfrak Q$, see \cite[\S 107]{Ca}. Motivated by Cayley's result, the problem of giving conditions on seven lines such that they are contained in a quartic surface received considerable attention in the late 1920s and early 1930s.   Perhaps the deepest results on the subject are due to Todd \cite{Todd}. He proves that  there exists a quartic surface containing seven pairwise disjoint lines if and only if there exists a rational curve of degree $19$ meeting six of the lines in twelve points and the seventh in ten points. Moreover,  if the seven lines belong to a linear line complex (in particular if the seven lines are all involutive) then he proves that there is always a quartic surface containing them. We could not follow in detail the arguments of Todd, but we present a simple proof of his result below.
For our purposes it is convenient to state it as follows.

\begin{theorem}[Todd]\label{Todd}
Let $\Lambda=\{\Lambda_1,...,\Lambda_7\}\subset\mathbb C^4$ be a set of seven Lagrangian planes.
Then $\Lambda$ is contained in a quartic.
\end{theorem}

Notice that Theorem \ref{Todd} implies that seven lines in an arbitrary line complex are
contained in a quartic. Indeed, up to projective automorphisms there are only two linear line complexes: the
one defined by the projectivazation of Lagrangian planes, and a degeneration of it consisting of all lines
which intersect another given line. The result for lines in an arbitrary linear line complex
follows from the semi-continuity of the dimension of the linear system of quartics containing
seven lines when the lines vary.

\subsection{Proof of Theorem \ref{Todd}}

Given seven involutive lines in $\mathbb P^3$ we  seek   a quartic
containing them.  We will do this by finding a non-trivial section $s\in H^0(\mathbb P^3,\mathcal O_{\mathbb P^3}(4))$ such that the zero locus of $s$ contains them.

Let $X=\sum_{j=1}^4z_j\frac{\partial}{\partial z_j}$ be the Euler (radial) vector field.
Then $\ell\subset\mathbb P^3$ is involutive if and only
the $1$-form  $\sigma:=\omega \lrcorner X$ vanishes identically when pulled-back to $\pi^* \ell$.
If we twist the dual of the  Euler exact sequence  (\cite[page 409]{GH})
\[
0 \to \Omega^1_{\mathbb P^3} \to \bigoplus_{i=0}^3 \mathcal O_{\mathbb P^3}(-1) \to \mathcal O_{\mathbb P^3} \to 0 ,
\]
by $\mathcal O_{\mathbb P^3}(2)$ we see that
$\sigma$ can be interpreted as a section $\sigma\in\Gamma(\mathbb P^3,\Omega^1_{\mathbb P^3}(2))$.
Thus a line $\ell$ is involutive if and only if it is tangent to the subbundle $\mathcal C:=Ker(\sigma)\subset T\mathbb P^3$, which is called a contact distribution. Considering the exact sequence
\begin{equation}\label{E:contact}
0 \to \mathcal C \to T\mathbb P^3 \to \mathcal O_{\mathbb P^3}(2) \to 0
\end{equation}
we obtain
\begin{equation}\label{E:detbundle}
det(\mathcal C)\simeq \det(T\mathbb P^3) \otimes \mathcal O_{\mathbb P^3}(-2) \simeq \mathcal O_{\mathbb P^3}(2).
\end{equation}

We want a lower estimate for the dimension of $H^0(\mathbb P^3,\mathcal C(1))$, the space of vector fields with coefficients in $\mathcal O_{\mathbb P^3}(1)$ which are tangent to the distribution. For this we consider the
exact sequence (\ref{E:contact}) twisted by $\mathcal O_{\mathbb P^3}(1)$.
Clearly we have the inequality
$$
\dim H^0(\mathbb P^3,\mathcal C(1) ) \ge \dim H^0(\mathbb P^3,T\mathbb P^3(1)) -  \dim H^0(\mathbb P^3,\mathcal O_{\mathbb P^3}(3)).
$$
The vector space $H^0(\mathbb P^3, \mathcal O_{\mathbb P^3}(3))$ is nothing but the vector space of cubic homogeneous polynomials in $4$ variables, and therefore has dimension
$20$. To determine the dimension $H^0(\mathbb P^3,T\mathbb P^3(1))$, first twist the Euler exact sequence \cite[page 409]{GH} to obtain
\[
0 \to \mathcal O_{\mathbb P^3}(1) \to \bigoplus_{i=0}^3 \mathcal O_{\mathbb P^3}(2) \to T \mathbb P^3(1) \to 0 ,
\]
and then look at the long exact sequence in cohomology. Since $h^1(\mathbb P^3, \mathcal O_{\mathbb P^3}(1))=0$ by \cite[Theorem 5.1 of Chapter III]{Hartshorne},
we have that
\[
\dim H^0(\mathbb P^3, T \mathbb P^3(1)) = 4 \dim H^0(\mathbb P^3, \mathcal O_{\mathbb P^3}(2)) - \dim H^0(\mathbb P^3, \mathcal O_{\mathbb P^3}(1)) = 40 - 4 = 36 \, .
\]
We conclude that
$\dim H^0(\mathbb P^3,\mathcal C(1))\geq 16$.

We now want to show that at least
two linearly independent sections of $H^0(\mathbb P^3,\mathcal C(1))$ leave
the seven lines invariant.   For this we note that   $v \in H^0(\mathbb P^3,\mathcal C(1))$  leaves a line $\ell$ invariant
if and only if it is in the kernel of the map
\begin{equation}\label{linearmap}
H^0(\ell,\mathcal C|_\ell(1))\rightarrow H^0(\ell,\mathcal N(1)),
\end{equation}
where $\mathcal N= \mathcal C|_\ell / T\ell$. Since both $T\ell$ and $\det \mathcal C|_\ell$ are isomorphic to $\mathcal O_{\ell}(2)$
 we see that $\mathcal N$ is the trivial line bundle on $\ell$, i.e., $\mathcal N=\mathcal O_{\ell}$.
 So (\ref{linearmap}) is the same as
\[
H^0(\ell,\mathcal C|_\ell(1))\rightarrow H^0(\ell,\mathcal O_\ell(1)).
\]
Since $\mathrm{dim} H^0(\ell,\mathcal O_\ell(1))=2$ it follows that the codimension
of the space of vector fields leaving $\ell$ invariant is at most $2$; hence
the vector space leaving the total of seven lines invariant has
dimension greater than or equal to two.   So let $v_1, v_2\in H^0(\mathbb P^3,\mathcal C(1))$
be linearly independent leaving the seven lines invariant.
If $v_1\wedge v_2$ is
not identically zero then $v_1\wedge v_2\in H^0(\mathbb P^3,\mathcal O_{\mathbb P^3}(4))$ is the sought quartic containing
the seven lines.

Aiming at contradiction let us assume that $v_1 \wedge v_2 = 0$. Thus
$v_1 = f v_2$ for some non-constant rational function $f \in \mathbb C(\mathbb P^3)$. It follows that
$v_1$  vanishes on the hypersurface $Z=\{ f=0\}$ while $v_2$ vanishes on the hypersurface $P=\{f= \infty\}$. After dividing $v_1$
by the equation of $Z$ we   get a section $w \in H^0(\mathbb P^3,\mathcal C(-a))$ for  $a = \deg Z - 1$. Since $H^0(\mathbb P^3, T\mathbb P^3(k))=0$ is zero
for $k < -1$ (look at the Euler sequence) we conclude that $\deg Z \in \{ 1, 2\}$.

If $\deg Z=2$ then we obtain a non-zero section $w \in H^0(\mathbb P^3 , \mathcal C(-1))$. But this leads to contradiction since
$H^0(\mathbb P^3 , \mathcal C(-1))=0$ because the map induced by (\ref{E:contact})
\begin{align*}
H^0(\mathbb P^3,T\mathbb P^3(-1)) &\longrightarrow H^0(\mathbb P^3,\mathcal O_{\mathbb P^3}(1)), \\
    \sum_{j=1}^{4} \lambda_j \frac{\partial}{\partial z_j} &\mapsto   \lambda_2 z_1 - \lambda_1 z_2 + \lambda_4 z_3 - \lambda_3 z_4
\end{align*}
is clearly an isomorphism.

If $\deg Z=1$ then $Z$ is a hyperplane containing at most one of the involutive lines (the lines
are disjoint) and $w \in H^0(\mathbb P^3,\mathcal C)$ is a global holomorphic vector field tangent to the contact distribution leaving invariant six distinct lines.
The global holomorphic vector fields  on $\mathbb P^3$ are in one to one correspondence with the elements of $\mathfrak{sl}(4,\mathbb C)$. To every $4\times 4$ matrix $A=(a_{ij})$ of trace zero
we associate the vector field $w_A = \sum_{i, j=1}^4 a_{ij} z_i \frac{\partial }{\partial z_j}$. Under this identification, the singular points of $w_A$
correspond to the eigenvectors of $A$, and the lines left invariant by $w_A$ correspond to the two-dimensional subspaces left invariant by $A$. Therefore the matrix corresponding
to $w$ leaves invariant $6$ generic two-dimensional subspaces of $\mathbb C^4$, and because of that must be a multiple of the identity. Since it has trace zero, it follows that $w=0$. Contradiction. \qed

\section{Hartshorne--Hirschowitz Theorem for involutive lines}\label{section:six}

Here we are interested in the question: What is the dimension of the linear system of surfaces of degree $d$
containing a given finite set of involutive lines? If we consider  a union $\Sigma$ of $r$ pairwise disjoint lines in $\mathbb P^3$ then
the degree $d$ surfaces in $\mathbb P^3$ containing $\Sigma$ can be identified with the projectivization of the kernel of
the restriction morphism
\[
H^0(\mathbb P^3,\mathcal O_{\mathbb P^3}(d)) \longrightarrow H^0(\Sigma, \mathcal O_{\Sigma}(d)) \, .
\]
Hartshorne and Hirschowitz \cite{HH1,HH2} proved that a general set of lines imposes independent conditions
on the linear system of degree $d$ surfaces \cite{HH1}, in other words  if the $r$ lines in $\Sigma$ are in general position and $d\ge0$, then the restriction morphism
above is a linear map of  maximal rank.

Theorem \ref{Todd}  implies that the same does not hold for $d=4$ and seven involutive lines in general position, but it turns out that
this is the only forbidden pair.

\begin{theorem}\label{T:HH}
Let $\Sigma$ be a union of $r$ involutive lines in $\mathbb P^3$ in general position and $d\ge 0$ be an integer.
If $(r,d) \neq (7,4)$  then  the restriction map
\[
H^0 (\mathbb P^3, \mathcal O_{\mathbb P^3}(d)) \to H^0(\Sigma,\mathcal O_{\Sigma}(d))
\]
is of maximal rank.
\end{theorem}

Despite the existence of the exceptional case, the proof of Theorem \ref{T:HH} contains no
novelty when compared with the proof of Hartshorne-Hirschowitz result presented in \cite{HH1}, except
for some minor extra complications coming from the speciality of the intersection of involutive lines
with quadrics. Indeed, the proof presented in this section follows \cite{HH1} word-by-word most of the
time.

The linear system of hypersurfaces in $\mathbb P^5$ containing arbitrary number of  planes  (involutive or not)
does not seem to be studied so far, cf. \cite{CCG}.

\subsection{Special quadrics}\label{S:specialquadrics}

Let $Q \subset \mathbb P^3$ be a quadric. The restriction of the contact form $\omega \in H^0(\mathbb P^3, \Omega^1_{\mathbb P^3}(2))$
determines a foliation on $Q = \mathbb P^1 \times \mathbb P^1$.
The induced foliation  $\mathcal F_Q$ on a general $Q$ will have
normal bundle equal to $N \mathcal F_Q = \mathcal O_{\mathbb P^1 \times \mathbb P^1}(2,2)$ and trivial tangent bundle. In this case, one or two lines  from
each ruling of $Q$ will be involutive. We will say that a smooth quadric $Q  \subset \mathbb P^3$ is a \emph{special quadric} if all the lines of one of the rulings
of $Q$ are involutive. This is equivalent to requiring that three lines of one of the rulings are involutive.

The intersections of involutive lines with a special quadric are described by a classical result of Chasles (cf. \cite[Theorem 10.2.10]{Dolgachev}):
for any given  special quadric $Q$, there exists an automorphism $\sigma: \mathbb P^1 \to \mathbb P^1$ such that the involutive lines are
exactly one of the rulings (say the vertical ruling) of $Q$, or intersect the horizontal ruling of $Q$ at heights given by  an orbit of $\sigma$.

\subsection{Reduction}

The proof of Theorem \ref{T:HH} will follow step-by-step the proof of the analogous statement for lines (instead of involutive lines)
established by Hartshorne and Hirschowitz. The main part of the proof consists of an induction argument of the statement ($H_d$) below, for $d \neq 4$.

\medskip

\noindent{\bf Statement ($H_d$): }
Let
\[
r = \left\lfloor\frac{1}{d+1} \binom{d+3}{3} \right\rfloor \text{ and } q = (d+1) \left( \frac{1}{d+1} \binom{d+3}{3} - r \right) \, .
\]
Then there exists a scheme $\Sigma \subset \mathbb P^3$ given by the union of $r$ involutive lines and $q$ points contained in another
 involutive line such that the restriction map
\[
H^0 (\mathbb P^3, \mathcal O_{\mathbb P^3}(d)) \to H^0(\Sigma,\mathcal O_{\Sigma}(d))
\]
is bijective.

\begin{lemma}
The assertions $(H_1), (H_2), (H_3), (H_6)$, and $(H_7)$  hold true.
\end{lemma}
\begin{proof}
The assertions $(H_1), (H_2),$ and $(H_3)$  are elementary. Assertions $(H_6)$ and $(H_7)$ have
been checked with the help of computer algebra, analogously to Section \ref{section:three}.
\end{proof}

\subsection{Inductive steps}

\begin{proposition}
If $d=0 \mod 3$ and $d\ge 3$  then $(H_{d-2})$ implies $(H_d)$.
\end{proposition}
\begin{proof}
Write $d=3k$, $k\ge1$. To prove $(H_d)$ we have to find $Y$, a union of $r= \frac{(k+1)(3k+2)}{2}$ involutive lines,  not contained in a surface
of degree $d$. We consider $Y = Y' \cup Y''$ where $Y'$ is the union of $2k+1$ involutive lines contained in a quadric $Q$ (all belonging to the same family of lines in $Q$), and $Y''$ is
a union of $\frac{k(3k+1)}{2}$ involutive lines  intersecting $Q$ transversely.

Suppose there exists a surface $F$ of degree $d$ containing $Y$. If $F$ does not contain $Q$, then its intersection with $Q$ is a curve
of type $(3k,3k)$ containing $Y'$ and the $k(3k+1)$ points of  $Y''\cap Q$. Therefore $F\cap Q$ is the union of $Y'$ and a curve $C'$ of bidegree $(k-1,3k)$
containing the points of $Y''\cap Q$. Since $h^0(Q,\mathcal O_Q(k-1,3k) = k(3k+1)$ we expect that there exists no such curve $C'$ if the points of $Y''$ are in general position.
Lemma \ref{L:pontos} below guarantees that this is the case, and implies that our surface $F$ must contain $Q$. Taking $Q$ out of $F$ we obtain a surface of degree $d-2$
containing $Y''$. If $(H_{d-2})$ is true then we can start with $Y''$ not contained in any surface of degree $d-2$, and conclude that ($H_d$) holds true.
\end{proof}

\begin{lemma}\label{L:pontos}
If $k\ge 1$, $Y''$ is a union of $\frac{k(3k+1)}{2}$ involutive lines in general position, and $Q$ is any smooth quadric, then there is no curve of bidegree $(k-1,3k)$
in $Q$ containing $Y''\cap Q$.
\end{lemma}
\begin{proof}
The set of $\frac{k(3k+1)}{2}$ involutive lines in general position satisfying the conclusion of the lemma is  open, and we only need to prove that it is non-empty.

If $k$ is even then we can choose the $\frac{k(3k+1)}{2}$ involutive lines
in such way that they intersect $Q$ in $k(3k+1)$ points distributed over $k$ distinct involutive lines of bidegree $(1,0)$ on $Q$, each of these involutive lines containing $3k+1$ points.
A curve of bidegree $(k-1,3k)$ containing all these points would also have to contain the $k$ involutive lines, which is impossible for a curve of this bidegree. This contradiction proves the lemma when $k$ is even.

If $k$ is odd then we can distribute  the $k(3k+1)$ points of intersection with $Q$ over $(k-1)$ involutive lines with each of them containing $3k+1$ points, and the extra $3k+1$ can be assumed to lie on $3k+1$ distinct lines of the non-involutive ruling of $q$. Clearly, a curve of bidegree $(k-1,3k)$  cannot pass through all these points, and the lemma follows  for $k$ odd.
\end{proof}

\begin{proposition}
If $d=2 \mod 3$ and $d\ge 3$  then $(H_{d-2})$ implies $(H_d)$.
\end{proposition}
\begin{proof}
Write $d=3k+2$, $k\ge 1$. Now we have to find $Y$, a union of $r=\frac{(k+1)(3k+6)}{2}$ involutive lines and $q=k+1$  points contained in another involutive line, so that $Y$ is not contained in a surface of degree $d$. Set $Y'$ equal to the union of $2k+2$  involutive lines contained in a quadric $Q$ and $k+1$ points contained in another involutive line of $Q$, with all the $2k +3$ lines contained in the same family.

If $k$ is even then set  $Y''$ equal to  the union of $\frac{(k+1)(3k+2)}{2}$ involutive lines not in $Q$, such that the $(k+1)(3k+2)$ points of intersection of these lines with $Q$
can be written as a disjoint union $A \cup B$ such that
\begin{enumerate}
\item the set $A$ consists of $k(3k+3)$ points contained in $k$ distinct involutive lines $\ell_1, \ldots, \ell_k$ of the same family mentioned before; and each of these $k$ lines contains exactly $3k+3$ of these points;
\item the set $B$ consists of the other $2k+2$ points and together with the $k+1$ points of $Y'$, form a set of $3k+3$ points that do not contain
 two points belonging to the same line belonging to the other family of lines of $Q$.
\end{enumerate}
If $S$ is a surface of degree $d$ containing $Y = Y' \cup Y''$, then either $S$ contains $Q$, or the intersection of $S$ with $Q$ will be equal to
the union of the lines in $Y'$ together with a curve of bidegree $(k,3k+2)$ containing the union of $A$, $B$, and  the $k+1$ points in $Y'$.  A curve $\Gamma$ of bidegree $(k,3k+2)$
containing $A$ must contain $\ell_1, \ldots, \ell_k$. Therefore $\Gamma - \cup_{i=1}^k \ell_i$ is a curve of bidegree $(0,3k+2)$ containing the points of $B$ and the $k+1$ points
in $Y'$. Since no two of these points are contained in a line of the other family we arrive at a contradiction, showing that $S$ must contain $Q$. Since this holds for this particular choice of $Y''$, it also holds true for a generic choice of $Y''$. Hence we conclude the proof of the proposition when $d$ is even.

Suppose now that $d$ (and thus also $k$) is odd, and set  $Y''$ equal to  the union of $\frac{(k+1)(3k+2)}{2}$ involutive lines not in $Q$, such that the $(k+1)(3k+2)$ points of intersection of these lines with $Q$
 can be written as a disjoint union $A \cup B \cup C$ such that
\begin{enumerate}
\item the set $A$ consists of $(k-1)(3k+3)$ points contained in $k-1$ distinct involutive lines $\ell_1, \ldots, \ell_{k-1}$ of the same family mentioned before; and each of these $k-1$ lines contains exactly $3k+3$ of these points;
\item the set $B$ consists of $2k+2$ points contained in the line $\ell_0$;
\item the set $C$ consists of $(2k+2) + k+1=3k+3$ points which do not contain
 two points belonging to the same line belonging to the  other family of lines of $Q$.
\end{enumerate}
If $S$ is a surface of degree $d$ containing $Y = Y' \cup Y''$ and not containing the quadric $Q$ then the intersection of $S$ with $Q$ will be equal to
the union of the lines in $Y'$ together with a curve $\Gamma$ of bidegree $(k,3k+2)$ containing the union of $A$, $B$, $C$ and  the $k+1$ points in $Y'$.
The curve $\Gamma$ must contain the $k-1$ lines $\ell_1, \ldots, \ell_{k-1}$, as well as the line $\ell_0$, since each of these lines contains $3k+3$ points and
$\Gamma$ has bidegree $(k,3k+2)$.   It follows that $\Gamma \setminus \cup_{i=0}^{k-1} \ell_i$ is the union of $(3k+2)$ points of the other family
and cannot contain all the points of $C$. As before, we arrived at a contradiction which shows that $S$ contains $Q$, and allow us to conclude.
\end{proof}

When $d = 1 \mod 3$ the reasoning of Hartshorne and Hirschowitz uses schemes with nilpotents elements and also the notion of \emph{residual schemes}.
If $H$ and $Y$ are subschemes of $\mathbb P^3$ then  $Z = res_H Y$,  the residual scheme of $Y$ in $H$, is the subscheme of
$\mathbb P^3$ with defining ideal given the kernel of the natural morphism
\[
\mathcal O_{\mathbb P^3} \longrightarrow Hom(  I_H, \mathcal O_Y) \, .
\]
If $H$ is a surface of degree $d$ then the residual scheme fits into the exact sequence
\[
0 \to \mathcal O_Z(-d) \to \mathcal O_Y \to \mathcal O_{Y\cap H} \to 0 \, .
\]

Let us consider  $Y_{\varepsilon}$, $\varepsilon \in \mathbb C^*$, the family
of disjoint lines $y=z=0$ and $x=z-\varepsilon=0$ on $\mathbb C^3$. For any $\varepsilon \in \mathbb C^*$
the two lines are involutive with respect to the contact structure determined by $xdy - ydx + dz$.
The ideal of $Y_{\epsilon}$ is
\[
(y,z) \cap (x,z-\varepsilon) = (xy, xz, y(z-\varepsilon), z(z-\varepsilon)) \, .
\]
The flat limit when $\varepsilon \to 0$ is the ideal $(xy,xz,yz, z^2)$. It represents the union of the lines $x=z=0$ and $y=z=0$
together with an immersed point at the origin. If $Y$ is the associated scheme and $H$ is the plane $y=x$, then the ideal of the
scheme $H\cap Y$ is $(y-x,x^2, xz, z^2)$. Thus $H\cap Y$ is a subscheme of $H$ supported at one point $p\in H$  and with structural
ring equal to $\mathcal O_{H,p} / \mathfrak m_{H,p}^2$. It is a scheme of length three or, in other words, a triple point in $H$.

The residual scheme has ideal $(z,xy)$ which is a reduced degenerate conic  at the plane $z=0$.

\begin{proposition}
If $d=3k+1$ with $k\ge 3$ then $(H_{d-4}) \implies (H'_{d-2}) \implies (H_d)$, where $(H'_{d-2})$ is the following

\noindent{\bf Statement ($H'_{d-2}$): } there exists a scheme $Y \subset \mathbb P^3$ which is the union of $\frac{(k-1)(3k-2)}{2}$ involutive
lines and $2k$ involutive reduced degenerate conics having their singular points at a special quadric $Q$, and such that the
natural morphism
\[
\rho(d-2) : H^0(\mathbb P^3, \mathcal O_{\mathbb P^3}(d-2))  \to   H^0(Y, \mathcal O_Y(d-2))
\]
is bijective.
\end{proposition}
\begin{proof}
We will start by proving that $(H'_{d-2})$ implies $(H_d)$. We have to find $\frac{(k+1)(3k+4)}{2}$ involutive lines such that $\rho(d)$ is
bijective. As the condition is open it suffices to produce one specialization of a union of disjoint involutive lines
having the sought property. We start by choosing a special quadric $Q$ and will take
$Y=Y'\cup Y''$, where $Y'$ is the union of $2k+1$ involutive lines on the involutive ruling of $Q$, and $Y''$  as the union of $\frac{(k-1)(3k-2)}{2}$
involutive lines in general position together with $2k$ involutive degenerate conics having nilpotent elements at singular points, which
are limits of pairs of disjoint involutive lines and have their singular points at $Q$.

As discussed above, the residual intersection of $Y''$ with $Q$ is nothing but $Y''_{red}$. We have an exact diagram
\[
\xymatrix{
0 \ar[r]& H^0(  \mathcal O_{\mathbb P^3}(d-2))   \ar^{\rho(d-2)}[d]\ar[r]  & H^0(  \mathcal O_{\mathbb P^3}(d))  \ar^{\rho(d)}[d]\ar[r]  &
H^0(  \mathcal O_{Q}(d))  \ar^{\alpha(d)} [d]\ar[r] & 0 \\
0 \ar[r]& H^0(  \mathcal O_{Y''_{red}}(d-2))  \ar[r]  &H^0( \mathcal O_Y(d)) \ar[r]  & H^0(  \mathcal O_{Y \cap Q}(d)) \ar[r] & 0 \\
}
\]
The hypothesis $(H'_{d-2})$ implies that the leftmost vertical arrow is bijective. Lemma \ref{L:a3k+1} implies that $\alpha(d)$ is bijective
if $d=3k+1$ with $k\ge 3$, and so the rightmost arrow is also bijective. Hence the same holds true for the middle arrow, and we have that $(H_d)$ holds true.

Let us now verify that $(H_{d-4})$  implies $(H'_{d-2})$. We want to find a scheme $Y$ which is the union of  $\frac{(k-1)(3k-2)}{2}$ involutive lines and $2k$ degenerate conics with singular points on the special quadric $Q$ such that $\rho(3k-1)$ is bijective. We will choose $Y$ with one of the lines of each degenerate conic belonging to the involutive fibration of $Q$, and one further simple line belonging to the same fibration. The other $2k$ lines belonging to the degenerate conics and the remaining  $\frac{(k-1)(3k-2)}{2} -1$  involutive lines will be chosen in general position. Therefore the residual scheme $Y''= res_Q Y$ is formed by $2k + \frac{(3k^2 - 5k)}{2}= \frac{k(3k-1)}{2}$ involutive lines in general position.

Let us consider the diagram below.
\[
\xymatrix{
0 \ar[r]& H^0(  \mathcal O_{\mathbb P^3}(d-4))   \ar^{\rho(d-4)}[d]\ar[r]  & H^0(  \mathcal O_{\mathbb P^3}(d-2))  \ar^{\rho(d-2)}[d]\ar[r]  &
H^0(  \mathcal O_{Q}(d-2))  \ar^{\alpha(d-2)} [d]\ar[r] & 0 \\
0 \ar[r]& H^0(  \mathcal O_{Y''_{red}}(d-4))  \ar[r]  &H^0( \mathcal O_Y(d-2)) \ar[r]  & H^0(  \mathcal O_{Y \cap Q}(d-2)) \ar[r] & 0 \\
}
\]
On the one hand  the leftmost vertical morphism $\rho(d-4)=\rho(3k-1)$ is bijective, by hypothesis. On the other hand the intersection of $Y$ with $Q$
is the union of $2k+1$ involutive lines and $2k + (3k^2 - 5k)$ points on $Q$. The $2k$ points coming from the intersection of smooth points of
the  conics are in general position in $Q$, and the remaining $(3k^2 - 5k)$ points are constrained only by Chasles Theorem, see Subsection \S\ref{S:specialquadrics} on special quadrics.
If $k$ is even then we can place $(3k^2- 6k)=3k(k-2)$ points over $k-2$ involutive  lines each containing $3k$ points. This suffices to ensure that an element
in the kernel of $\alpha(d-2)=\alpha(3k-1)$ is a product of $3k-1$ involutive lines and $3k-1$ non-involutive lines. But we can choose the remaining $3k$ points not contained
in any set of $3k-1$ non-involutive lines. This shows that $\alpha(d-2)$ is bijective for a general $Y$ when $d$ is even.  Similarly, if $d$ is odd then we
can place $(3k^2- 5k)=3k(k-2) + k$ over  $k-1$ involutive lines in such a way that $k-2$ involutive lines contains $3k$ points, and the remaining involutive line contains $k$
points in general position. If we place the $2k$ points coming from the degenerate conics in general position on this last line, the injectivity of $\alpha(d-2)$ follows
as before. To summarize, for a general $Y$ the morphism $\alpha(d-2)$ is also bijective. It follows that the middle vertical morphism is also bijective, i.e., $(H'_{d-2})$ holds true. \end{proof}

\begin{lemma}\label{L:a3k+1}
The map $\alpha(3k+1)$ is bijective when $k\ge 3$.
\end{lemma}
\begin{proof}
Notice that $Y\cap Q$ consists of $2k+1$ involutive lines on the involutive ruling of $Q$, since $Y' \subset Q$, and
$(k-1)(3k-2) + 4k$ simple points together with $2k$ triple points. If $\alpha(d)$ is not bijective, there exists a curve of bidegree $(d,d)= (3k+1,3k+1)$
containing schematically  all these points, triple points, and lines.

To prove that $\alpha(d)$ is bijective for a general $Y''$, it suffices, by semi-continuity, to prove it for a special $Y''$ which we now proceed to construct.

If $k$ is even and $k\ge 4$ then  we start by choosing $4$ involutives lines belonging to the involutive ruling of  $Q$ and distinct from the $2k+1$ involutive lines in $Y'$, say $\ell_1, \ell_2, \ell_3,$ and $\ell_4$.
In each of the $\ell_1$ and $\ell_2$ we place $k$ triple points and choose the corresponding degenerate involutive conics of $Y''$ in such  way that
each of them intersect each of the lines $\ell_3$ and $\ell_4$ in one  point. We put $k+2$ simple points in $\ell_1$ with the corresponding involutive lines
intersecting $\ell_2$. And finally we place $k+2$ simple points in $\ell_3$ with the corresponding involutive lines intersecting $\ell_4$. If we set $A$ equal to
the union of these $2k$ degenerate conics and $2(k+2)$ involutive lines, then the length of the restriction of  $A \cap Q$ to each of the lines $\ell_i$, $i=1,\ldots,4$,
is $3k+2$.  Now let $B$ be a union of $\frac{(k-4)(3k+2)}{2}$ involutive lines such that $B \cap Q$ is contained in $k-4$ lines of the involutive ruling
and each of these contains $3k+2$ points; and let $C$ be a union of general $\frac{k}{2} +1$ involutive lines. If we choose $Y'' = A \cup B \cup C$ then a curve $C$
of bi-degree $(d,d)=(3k+1,3k+1)$ containing $Y \cap Q = (Y' \cup Y '') \cap Q$ must contain the $2k+1$ lines in $Y'$, the $4$ involutive lines $\ell_1, \ldots, \ell_4$,
and the $(k-4)$ involutive lines containing the support of $B$. It follows that a hypothetical curve $C$ containing $Y\cap Q$ must be a product of $3k+1$ involutive lines
and $3k+1$ non-involutive lines of $Q$.

If we further assume that no two of the $2k$ triple points in $A\cap Q$ have support on the same line of the non-involutive ruling, and we assume the same for the $k+1$ points of intersection of $C$ with $Q$, then it follows that such a hypothetical curve cannot exists, since it would pass through the support of $2k$ triple points and $C\cap Q$. Therefore $\alpha(d)$ is bijective for $Y''$ sufficiently general
when $k$ is even and $k\ge 4$.

If $k$ is odd and $k\ge 3$ we will proceed similarly. We choose $3$ involutives lines belonging to the involutive ruling of  $Q$ and distinct from the $2k+1$ involutive lines in $Y'$, say $\ell_1, \ell_2,$ and $\ell_3$.  On the line the $\ell_1$ (resp. the line $\ell_2$) we place $k$ triple points and choose the corresponding degenerate involutive conics of $Y''$ in such  way that each of them intersects the lines $\ell_2$ (resp. $\ell_1$) and $\ell_3$ in one  point. We put $2$ simple points on $\ell_1$ with the corresponding involutive lines
intersecting $\ell_2$; and $k+2$ simple points on $\ell_3$. We set $A$ equal to the union of these $2k$ degenerate conics, and $k+4$ involutive lines.
If the triple points and simple points in $\ell_3$ are in general position then the support of the $2k$ triple points and the $k+2$ simple points outside of $\ell_3$ coming from the involutive lines intersecting $\ell_3$, will belong to $3k+2$ distinct lines of the non-involutive ruling of $Q$.
Let   $B$ be a union of $\frac{(k-3)(3k+2)}{2}$ involutive lines such that $B \cap Q$ is contained in $k-3$ lines of the involutive ruling
and each of these contains $3k+2$ points. If we set $Y'' = A \cup B$ then the same argument used for $k$ even implies that there is no curve $C$
of bidegree $(3k+1,3k+1)$ in $Q$  containing schematically $Y'' \cap Q$. We conclude that $\alpha(3k+1)$ is bijective  for $k\ge 3$.
\end{proof}

\subsection{Synthesis} We can summarize what we have proved so far in the following theorem.

\begin{theorem}\label{T:Hd}
If $d\neq4$ then statement $(H_d)$ holds true.
\end{theorem}

\subsection{Proof of Theorem \ref{T:HH}}

Set $r_0$ equal to $\lfloor \frac{1}{d+1}\binom{d+3}{3}\rfloor$.
If $r\le r_0$ and $d\neq 4$ then Theorem \ref{T:Hd} guarantees the existence of a scheme $Y$ which is the union of $r_0$ pairwise disjoint involutive lines
such that the restriction morphism $H^0(\mathbb P^3, \mathcal O_{\mathbb P^3}(d)) \to H^0(Y,\mathcal O_Y(d))$ is surjective. After taking out $r_0-r$ lines from
$Y$ we obtain a scheme $\Sigma$ for which the analogous morphism is still surjective.

If $r> r_0$  and $d\neq 4$ then Theorem \ref{T:Hd} gives a scheme $Y$ equal to the union of $r_0$ pairwise disjoint involutives lines and $ q = (d+1) \left( \frac{1}{d+1} \binom{d+3}{3} - r \right)$ contained in another involutive line disjoint from all the previous ones such that $H^0(\mathbb P^3, \mathcal O_{\mathbb P^3}(d)) \to H^0(Y,\mathcal O_Y(d))$ is a bijection. We take $\Sigma$ as the union of lines in $Y$ together with the line determined by the $q$ points and $r-r_0 -1$ involutive lines disjoint from the other ones, to obtain a scheme for which the restriction morphism in degree $d$ is injective.

If $d=4$ then to conclude the proof of Theorem \ref{T:HH} it remains to check that the restriction maps is surjective  for $r<7$, and injective for $r>7$. It is sufficient to check surjectiveness for $r=6$ (the remaining cases can be  obtained from this one by taking lines out), and injectiveness for $r=8$ (where the remaining cases can be obtained from this
one by adding lines). Both verifications have been done by computer.
\qed

\bibliographystyle{amsplain}

 \end{document}